\documentclass{article}

\usepackage{amsmath}
\usepackage{amssymb}
\usepackage{graphicx}
\usepackage{bm}
\usepackage{mathdots}
\usepackage{booktabs}
\usepackage{rotating}
\usepackage{listings}
\usepackage[ruled,linesnumbered]{algorithm2e}
\usepackage[a4paper,left=2.8cm,right=2.8cm,top=2.5cm,bottom=2.5cm]{geometry}
\usepackage{fancyhdr}
\usepackage[framemethod=tikz]{mdframed}
\newtheorem{theorem}{Theorem}[section]
\newtheorem{corollary}[theorem]{Corollary}
\newtheorem{proposition}[theorem]{Proposition}

\newtheorem{lemma}[theorem]{Lemma}

\makeatletter 
\@addtoreset{equation}{section}
\makeatother  

\newtheorem{remark}[theorem]{Remark}

\newenvironment{proof}{{\noindent\it Proof.}\quad}{\hfill $\square$\\}

\usepackage{url}
\usepackage{mathtools}
\usepackage{lipsum}
\usepackage{tcolorbox}

\begin{document}
\title{Hyperinterpolation beyond exact cubature: a spectral multiplier approach}

\author{Hao-Ning Wu\footnotemark[1]}

\renewcommand{\thefootnote}{\fnsymbol{footnote}}
\footnotetext[2]{Department of Mathematics, University of Georgia, Athens, GA 30602, USA (hnwu@uga.edu)}

\maketitle
\begin{abstract}
We study hyperinterpolation and its spectral multiplier variants on the sphere under weak cubature assumptions formulated through Sobolev discrepancy estimates. In contrast with classical hyperinterpolation theory, our framework does not require exact polynomial cubature formulas or Marcinkiewicz--Zygmund inequalities. The main idea is to interpret the discretization error as the action of a spectral multiplier operator on the cubature discrepancy measure. This viewpoint separates approximation properties of the underlying spectral operator from geometric properties of the sampling measure, leading to stable Sobolev approximation estimates under weak cubature assumptions. The resulting theory applies to a broad class of spectral approximation operators, including sharp spectral projections, compactly supported smooth filters, Bessel potential operators, and heat kernel operators. For sufficiently localized spectral multipliers, we additionally obtain uniform \(L^\infty\)-stability of the corresponding discrete approximation operators.
The results establish a direct connection between hyperinterpolation, Sobolev discrepancy, and quasi-Monte Carlo (QMC) designs, showing that stable approximation from scattered data can be achieved without exact polynomial reproduction.
\end{abstract}

\section{Introduction}

Let
\[
\mathbb S^d
:=
\{x\in\mathbb R^{d+1}:|x|=1\},
\qquad d\ge1,
\]
be the unit sphere equipped with the normalized surface measure \(\sigma\).
Denote by \(\mathcal H_\ell(\mathbb S^d)\) the space of spherical harmonics
of degree \(\ell\), and let
\[
\{Y_{\ell,k}:1\le k\le Z(d,\ell)\}
\]
be an orthonormal basis of \(\mathcal H_\ell(\mathbb S^d)\) in
\(L^2(\mathbb S^d)\), where
\[
Z(d,\ell)
=
(2\ell+d-1)
\frac{\Gamma(\ell+d-1)}
{\Gamma(d)\Gamma(\ell+1)}
\sim
\frac{2}{\Gamma(d)}
\ell^{d-1},
\qquad \ell\to\infty,
\]
denotes the dimension of \(\mathcal H_\ell(\mathbb S^d)\) and $\Gamma(\cdot)$ is the Gamma function.

The spherical harmonics are eigenfunctions of the Laplace--Beltrami operator:
\[
-\Delta_{\mathbb S^d}Y_{\ell,k}
=
\lambda_\ell Y_{\ell,k},
\qquad
\lambda_\ell=\ell(\ell+d-1).
\]
Let
\[
\Pi_n
:=
\bigoplus_{\ell=0}^n
\mathcal H_\ell(\mathbb S^d)
\]
denote the space of spherical harmonics of degree at most \(n\).
The orthogonal projection
\[
\mathcal P_n:L^2(\mathbb S^d)\to\Pi_n
\]
admits the representation
\begin{equation}\label{equ:orthogonalprojection}
\mathcal P_n f(x)
=
\int_{\mathbb S^d}
G_n(x,y)f(y)\,d\sigma(y),
\end{equation}
where
\begin{equation}\label{equ:projectionkernel}
G_n(x,y)
=
\sum_{\ell=0}^n
\sum_{k=1}^{Z(d,\ell)}
Y_{\ell,k}(x)Y_{\ell,k}(y)
\end{equation}
is the reproducing kernel of \(\Pi_n\).

Hyperinterpolation, introduced by Sloan in \cite{sloan1995polynomial}, replaces the integral in the orthogonal projection \eqref{equ:orthogonalprojection} by a cubature rule. Given sampling points
$\mathcal{X}_m=\{x_j\}_{j=1}^m\subset\mathbb S^d$
and positive weights $\{w_j\}_{j=1}^m$,
the hyperinterpolation operator is defined by
\begin{equation}\label{equ:hyperinterpolation}
\mathcal L_n f(x)
=
\sum_{j=1}^m
w_jf(x_j)G_n(x,x_j).
\end{equation}
This construction yields a fully discrete approximation operator while preserving many approximation properties of the orthogonal projector.

Classical hyperinterpolation theory relies heavily on strong algebraic assumptions on the underlying cubature rule. Typically one assumes polynomial exactness up to sufficiently high degree of $2n$, as realized for instance by spherical $t$-designs \cite{delsarte1991geometriae} and related cubature constructions. Under such assumptions, the hyperinterpolation operator \eqref{equ:hyperinterpolation} inherits stability and approximation properties from the continuous projection operator \eqref{equ:orthogonalprojection}. 
An important further development is the use of
Marcinkiewicz--Zygmund (MZ) inequalities \cite{filbir2011marcinkiewicz,mhaskar2001spherical}
in the analysis of hyperinterpolation and related approximation schemes;
see, e.g.,
\cite{an2022quadrature,an2024bypassing,an2024hyperinterpolation}.
Instead of exact polynomial cubature,
the MZ framework assumes discrete norm equivalences on polynomial spaces.
This substantially enlarges the class of admissible sampling sets
and provides a flexible framework for approximation from scattered data
on the sphere.
We also refer the reader to the recent survey
\cite{MR5019366}
for further developments in hyperinterpolation theory.

Nevertheless, both exact cubature and MZ inequalities remain fundamentally algebraic assumptions. In many practical situations one only has access to scattered quasi-uniform point sets without exact polynomial reproduction properties or stable discrete norm equivalences. This naturally raises the following question:

\medskip

\emph{To what extent can hyperinterpolation be developed under substantially weaker assumptions on the sampling measure?}

\medskip

In the present paper, we study this problem from a different perspective based on
Sobolev discrepancy and spectral multiplier operators.
Let
\begin{equation}\label{equ:discretemeasure}
\mu_m
=
\sum_{j=1}^m w_j \delta_{x_j}
\end{equation}
be the discrete sampling measure associated with the point set $\mathcal{X}_m$, where $\delta_{x_j}$ denotes the Dirac mass at $x_j$, and define
the discrepancy measure
\begin{equation}\label{equ:discrepancyerrror}
\nu_m
:=
\mu_m-\sigma.
\end{equation}
Instead of assuming exact polynomial cubature or MZ inequalities, we only require
the weak cubature condition
\begin{equation}\label{equ:introcondition}
\|\nu_m\|_{H^{-r}}
\le
C\delta_m^r,
\end{equation}
where \(\delta_m\) denotes the mesh norm of the sampling set and $C>0$ is a constant independent of $m$.
Equivalently, this condition controls the worst-case integration error
over the Sobolev space \(H^r(\mathbb{S}^d)\).
Such discrepancy estimates arise naturally in the theory of
QMC designs \cite{MR3365840,MR3246811},
whose connection with the present framework will be discussed in
Section~\ref{sec:connection_with_spherical_qmc_designs}.

More generally, let
\[
\eta_n:[0,\infty)\to\mathbb R
\]
be a family of spectral multipliers, and define the associated spectral approximation operator
\begin{equation}\label{equ:spectralapproximation}
\mathcal T_n f(x)
:=
\int_{\mathbb S^d}
\Phi_n(x,y)f(y)\,d\sigma(y),
\end{equation}
where
\[
\Phi_n(x,y)
=
\sum_{\ell=0}^\infty
\eta_n(\lambda_\ell)
\sum_{k=1}^{Z(d,\ell)}
Y_{\ell,k}(x)Y_{\ell,k}(y).
\]
The corresponding discrete approximation operator is thus defined by
\begin{equation}\label{equ:spectrallyfilteredhyperinterpolation}
\mathcal F_n f(x)
=
\sum_{j=1}^m
w_jf(x_j)\Phi_n(x,x_j).
\end{equation}
The operator \eqref{equ:spectrallyfilteredhyperinterpolation} provides a unified framework for several classical approximation schemes on the sphere.
Indeed, choosing
\[
\eta_n(\lambda_\ell)
=
\mathbf 1_{\{\ell\le n\}},
\]
one recovers the classical hyperinterpolation operator associated with the sharp spectral projection, proposed in \cite{sloan1995polynomial}.
On the other hand, choosing
\[
\eta_n(\lambda_\ell)=h(\ell/n),
\]
where \(h\) is a compactly supported smooth filter,
yields the filtered hyperinterpolation operators studied in
\cite{MR3739962,MR4226998,MR4462410,MR2875103,sloan2012filtered,MR3685994}. More generally, the present framework also accommodates non-compactly supported spectral multipliers,
including Bessel potential operators and heat kernel operators.
Accordingly, we shall refer to
\eqref{equ:spectrallyfilteredhyperinterpolation}
as a \emph{spectral multiplier hyperinterpolation operator}, since it is obtained by discretizing
a spectral multiplier operator through cubature.

The key observation is that the discretization error admits the representation
\[
\mathcal{T}_n f-\mathcal{F}_n f
=
-\mathcal{T}_n(f\nu_m).
\]
Thus the approximation problem may be viewed as the action of the spectral multiplier
operator on the cubature discrepancy measure.
From this viewpoint, the analysis separates naturally into three independent
components:

\begin{itemize}
\item approximation properties of the continuous spectral operator;
\item smoothing estimates of the spectral multiplier operator;
\item Sobolev discrepancy estimates of the sampling measure.
\end{itemize}

Under these assumptions, we establish Sobolev approximation estimates for the spectral multiplier hyperinterpolation operator \eqref{equ:spectrallyfilteredhyperinterpolation} under the weak cubature assumption \eqref{equ:introcondition}.
The framework applies not only to classical sharp spectral projection,
but also to smooth spectral multipliers including compactly supported filters,
Bessel potential operators, and heat kernel operators.
For sufficiently localized spectral multipliers, we additionally obtain uniform
\(L^\infty\)-stability of the discrete approximation operator.
Thus spatial localization appears naturally as a mechanism for uniform stability,
while the Sobolev approximation theory itself is driven primarily by smoothing
estimates acting on the discrepancy measure.

\textbf{Notation.} Throughout the paper, \(C\) denotes a generic positive constant
whose value may change from line to line.
Unless otherwise specified, all constants $C$ are independent of
the discretization parameters \(m\), \(n\), and the function \(f\).
For two nonnegative quantities \(A\) and \(B\),
we write $A \lesssim B$
if there exists a constant \(C>0\) such that
$A \le C B$,
and write
$A \sim B$
if both \(A\lesssim B\) and \(B\lesssim A\) hold.

\textbf{Organization.}
The paper is organized as follows.
Section~\ref{equ:spherical_geometry_and_harmonic_analysis}
introduces the geometric setting, spherical harmonic analysis,
and Sobolev spaces on the sphere.
Section~\ref{sec:main_results}
contains the main approximation and stability results.
Section~\ref{sec:examples_of_spectrally_localized_kernels}
presents several important classes of spectral multipliers,
including sharp spectral projection, compactly supported smooth filters,
Bessel potential multipliers, and heat kernel operators.
Finally,
Section~\ref{sec:connection_with_spherical_qmc_designs}
discusses the connection between the present framework
and spherical QMC designs.

\section{Spherical geometry and harmonic analysis}\label{equ:spherical_geometry_and_harmonic_analysis}

In this section we collect several standard facts on the sphere. These ingredients form the analytic foundation for the approximation results proved later.

\subsection{Sampling geometry}

Let
\[
\mathcal{X}_m=\{x_j\}_{j=1}^m
\subset\mathbb S^d
\]
be a finite set of sampling points. The geodesic distance between any $x,y\in\mathbb{S}^d$ is defined by
\[
\rho(x,y):=\arccos(x\cdot y).
\]
 The mesh norm of $\mathcal{X}_m$ is thus defined by
\begin{equation}\label{equ:meshnorm}
\delta_m
:=
\sup_{x\in\mathbb S^d}
\min_{1\le j\le m}
\rho(x,x_j),
\end{equation}
and the separation radius is
\[
\gamma_m
:=
\frac12
\min_{i\neq j}
\rho(x_i,x_j).
\]
We say that the family $\{\mathcal{X}_m\}$ is quasi-uniform if there exists a constant $C>0$ independent of $m$ such that
\[
\delta_m\le C \gamma_m.
\]
Associated with the sampling set is a family of positive weights $\{w_j\}_{j=1}^m$,
which is assumed to satisfy 
\[
\sum_{j=1}^m w_j
=
\sigma(\mathbb S^d) = 1.
\]

\subsection{Spherical harmonics and Sobolev spaces}

Let $L^2(\mathbb{S}^d)$ denote the Hilbert space of all square-integrable functions on $\mathbb{S}^d$ with the inner product
$$
\langle f, g\rangle:=\int_{\mathbb{S}^d} f(x) g(x) d \sigma(x)
$$
and the induced norm 
\[\|f\|_{L^2}:=\sqrt{\langle f, f\rangle}=\sum_{\ell=0}^\infty
\sum_{k=1}^{Z(d,\ell)}
|\widehat f_{\ell,k}|^2,\] 
where
\[
\widehat f_{\ell,k}
=
\int_{\mathbb S^d}
f(x)Y_{\ell,k}(x)\,d\sigma(x).
\]
By ${C}(\mathbb{S}^d)$ we denote the space of continuous functions on $\mathbb{S}^d$, endowed with the uniform norm 
\[\|f\|_{L^\infty}:=\text{ess}\sup _{x \in \mathbb{S}^d}|f(x)|.\]
For $s\in\mathbb R$, the Sobolev space $H^s(\mathbb S^d)$
consists of all distributions $f$ such that
\[
\|f\|_{H^s}^2
:=
\sum_{\ell=0}^\infty
(1+\lambda_\ell)^s
\sum_{k=1}^{Z(d,\ell)}
|\widehat f_{\ell,k}|^2
<
\infty.
\]
In particular, for \(s>0\), the negative Sobolev space
\[
H^{-s}(\mathbb S^d)
:=
(H^s(\mathbb S^d))'
\]
can be also identified as the dual space of \(H^s(\mathbb S^d)\) with respect to the \(L^2\)-pairing, equipped with the norm 
\[
\|u\|_{H^{-s}}
=
\sup_{\|g\|_{H^s}\le1}
|\langle u,g\rangle|.
\]
We shall repeatedly use the Sobolev embedding theorem: if $s>d/2$, then
\[
H^s(\mathbb S^d)
\hookrightarrow
C(\mathbb S^d).
\]

\section{Main results}\label{sec:main_results}

Recall the discrepancy measure \eqref{equ:discrepancyerrror},
\[
\nu_m=\mu_m-\sigma,
\]
with $\mu_m$ defined by \eqref{equ:discretemeasure}. Since the discrete sampling measure $\mu_m$ contains Dirac masses,
we shall assume $r>d/2$ so that 
\[\nu_m \in H^{-r}(\mathbb{S}^d).\] 
Throughout this paper, we assume that there exists $r>d/2$ such that
\begin{equation}\label{equ:pointassumption}
\|\nu_m\|_{H^{-r}}
\le
C\delta_m^r,
\end{equation}
where $\delta_m$ is the mesh norm \eqref{equ:meshnorm} of the sampling set $\mathcal{X}_m$.
Equivalently, this assumption \eqref{equ:pointassumption} amounts to
\[
\sup_{\|g\|_{H^r}\le1}
\left|
\int_{\mathbb S^d}g\,d\sigma
-
\sum_{j=1}^m
w_jg(x_j)
\right|
\le
C\delta_m^r.
\]
This assumption controls only the cubature error of the integration functional and does not impose any exact polynomial reproduction property or discrete norm equivalence on polynomial spaces.

\subsection{Discretization as a distributional perturbation}

We first recall a
standard multiplier property of Sobolev spaces which allows us
to interpret products of the discrepancy measure $\nu_m$ with smooth
functions.
\begin{lemma}\label{lem:estimate}
Let $r> d/2$ and $s>d/2+r$. Then multiplication by
$f\in H^s(\mathbb S^d)$ defines a bounded operator on
$H^r(\mathbb S^d)$.
More precisely,
\[
\|fg\|_{H^r}
\lesssim
\|f\|_{H^s}
\|g\|_{H^r}
\]
for all $g\in H^r(\mathbb S^d)$.
Consequently, by duality,
\[
\|fu\|_{H^{-r}}
\lesssim
\|f\|_{H^s}
\|u\|_{H^{-r}}
\]
for all $u\in H^{-r}(\mathbb S^d)$.
\end{lemma}

\begin{proof}
Since \(s>d/2+r\), the Sobolev space \(H^s(S^d)\)
acts as a multiplier algebra on \(H^r(S^d)\):
\[
\|fg\|_{H^r}
\lesssim
\|f\|_{H^s}
\|g\|_{H^r}.
\]
The estimate on $H^{-r}$ follows by duality. Indeed, for \(u\in H^{-r}\) and \(v\in H^r\),
\[
|\langle fu,v\rangle|
=
|\langle u,fv\rangle|
\le
\|u\|_{H^{-r}}
\|fv\|_{H^r}
\lesssim
\|f\|_{H^s}
\|u\|_{H^{-r}}
\|v\|_{H^r}.
\]
Taking the supremum over
\(\|v\|_{H^r}\le1\)
gives the estimate.
\end{proof}

Recall the continuous spectral approximation operator
\(\mathcal T_n\)
and the discrete spectral multiplier hyperinterpolation operator
\(\mathcal F_n\)
 defined by
\eqref{equ:spectralapproximation}
and
\eqref{equ:spectrallyfilteredhyperinterpolation},
respectively.
The central observation of the present paper is that the discretization error
can be interpreted as the action of the spectral multiplier operator on the
cubature discrepancy measure.
Since \(r>d/2\), the discrepancy measure \eqref{equ:discrepancyerrror}
belongs to \(H^{-r}(\mathbb S^d)\).
Moreover, by Lemma~\ref{lem:estimate}, $f\nu_m\in H^{-r}(\mathbb S^d)$
whenever
$f\in H^s(\mathbb S^d)$ with
$s>r+d/2$.
A direct computation gives
\[
\begin{aligned}
\mathcal T_n f(x)-\mathcal F_n f(x)
&=
\int_{\mathbb S^d}
\Phi_n(x,y)f(y)\,d\sigma(y)
-
\sum_{j=1}^m
w_j f(x_j)\Phi_n(x,x_j)
\\
&=
\int_{\mathbb S^d}
\Phi_n(x,y)f(y)\,d(\sigma-\mu_m)(y)
\\
&=
-\mathcal T_n(f\nu_m)(x),
\end{aligned}
\]
where \(T_n\) is extended to distributions by
\[
T_n u(x):=\langle u,\Phi_n(x,\cdot)\rangle .
\]
Hence,
\begin{equation}\label{equ:identity}
\mathcal T_n f-\mathcal F_n f
=
-\mathcal T_n(f\nu_m).
\end{equation}
Therefore, the discretization error is reduced to estimating the action of the
spectral multiplier operator on the discrepancy distribution \(f\nu_m\).
From this viewpoint, the approximation problem separates naturally into two
independent components:
\begin{itemize}
\item Sobolev discrepancy estimates for the sampling measure;
\item smoothing properties of the spectral multiplier operator.
\end{itemize}

\subsection{Sobolev error bounds}

We now state the main approximation result.

\begin{theorem}[Sobolev approximation under weak cubature]\label{thm:main}
Let $r>d/2$, $0\le q<s$ and $s>d/2+r$. Let $f\in H^s(\mathbb{S}^d)$. Assume that the sampling measure satisfies
\begin{equation}
\|\nu_m\|_{H^{-r}}
\lesssim\delta_m^r.\tag{A1}
\end{equation}
Assume moreover that the continuous spectral multiplier operator $\mathcal{T}_n$ satisfies the approximation estimate
\[
\|f-\mathcal T_n f\|_{H^q}
\lesssim n^{-(s-q)}
\|f\|_{H^s}\tag{A2}
\]
and the smoothing estimate
\[
\|\mathcal T_n u\|_{H^q}
\lesssim n^{r+q}
\|u\|_{H^{-r}}\tag{A3}
\]
for all $u\in H^{-r}(\mathbb{S}^d)$.
Then
\[
\|f-\mathcal F_n f\|_{H^q}
\lesssim
\left(
n^{-(s-q)}
+
\delta_m^r n^{r+q}
\right)
\|f\|_{H^s}.
\]
\end{theorem}

\begin{proof}
We decompose the error as
\[
f-\mathcal F_n f
=
(f-\mathcal T_n f)
+
(\mathcal T_n f-\mathcal F_n f).
\]
The first term is controlled by the approximation property (A2) of the continuous spectral multiplier operator.
For the second term, we use the identity \eqref{equ:identity}.
Hence the smoothing estimate (A3) gives
\[
\|\mathcal T_n f-\mathcal F_n f\|_{H^q}
\lesssim n^{r+q}
\|f\nu_m\|_{H^{-r}}.
\]
Since $s>r+d/2$, Lemma \ref{lem:estimate} implies
\[
\|f\nu_m\|_{H^{-r}}
\lesssim
\|f\|_{H^s}
\|\nu_m\|_{H^{-r}}.
\]
Using the weak cubature assumption (A1),
we obtain
\[
\|\mathcal T_n f-\mathcal F_n f\|_{H^q}
\lesssim\delta_m^r n^{r+q}
\|f\|_{H^s}.
\]
Combining these estimates yields the desired error bound.
\end{proof}

The theorem reveals a natural balance between spectral approximation and cubature accuracy. The term
$n^{-(s-q)}$
corresponds to the approximation error of the continuous spectral multiplier operator, while
$\delta_m^r n^{r+q}$
measures the interaction between
weak cubature accuracy and the smoothing properties of the spectral multiplier operator.
Balancing the two contributions in the sense of 
\[n^{-(s-q)}\sim \delta_m^r n^{r+q}\]  
leads to the choice
\begin{equation}\label{equ:optimalchoice}
n
\sim
\delta_m^{-r/(s+r)}.
\end{equation}

\begin{corollary}[Balanced approximation rate]
Under the assumptions of Theorem \ref{thm:main} and the choice \eqref{equ:optimalchoice},
\[
\|f-\mathcal F_n f\|_{H^q}
\lesssim
\delta_m^{\frac{r(s-q)}{s+r}}
\|f\|_{H^s}.
\]
\end{corollary}

\begin{corollary}[$L^2$ approximation]
Under the assumptions of Theorem \ref{thm:main} and the choice \eqref{equ:optimalchoice},
\[
\|f-\mathcal F_n f\|_{L^2}
\lesssim
\delta_m^{\frac{rs}{s+r}}
\|f\|_{H^s}.
\]
\end{corollary}
\begin{remark}
The approximation rate above differs from the classical
\(L^2\)-based estimates for hyperinterpolation.
This reflects the fact that the present framework assumes only
weak Sobolev control of the cubature discrepancy measure and does
not rely on exact polynomial cubature formulas or
MZ stability estimates.
In particular, there is no underlying discrete
\(L^2\)-projection structure in the present setting.
Consequently, the discretization error cannot be removed
through algebraic cancellation and must instead be controlled
analytically through smoothing estimates acting on the discrepancy
measure.
\end{remark}

\begin{corollary}[$L^\infty$ approximation]
Under the assumptions of Theorem \ref{thm:main} and the choice \eqref{equ:optimalchoice}, assume additionally that
$q>d/2$.
Then
\[
\|f-\mathcal F_n f\|_{L^\infty}
\lesssim
\delta_m^{\frac{r(s-q)}{s+r}}
\|f\|_{H^s}.
\]
\end{corollary}

\subsection{$L^\infty$ stability}

A second important consequence of spectral localization is the
uniform stability of the discrete approximation operator.
This extends the filtered hyperinterpolation stability result of
Sloan and Womersley \cite{sloan2012filtered}
in two directions:
first, the present framework does not require exact cubature formulas;
second, it applies to a broader class of spectral multipliers
beyond compactly supported smooth filters,
including the examples discussed in
Section~\ref{sec:examples_of_spectrally_localized_kernels}.

The localization properties of $\Phi_n$ play a central role. Roughly speaking, smooth spectral multipliers generate kernels that are spatially localized.
Typical localization estimates take the form
\begin{equation}\label{equ:localization}
|\Phi_n(x,y)|
\le
C_Ln^d(1+n\rho(x,y))^{-L},
\end{equation}
where the decay exponent $L$ depends on the smoothness of the multiplier;
see, e.g.,
\cite{MR3060033,MR2673702,MR2475947,MR1826335,MR2237162}.

\begin{theorem}[$L^\infty$ stability]\label{thm:uniformstability}
Assume that the sampling sets are quasi-uniform and that
\[
0<w_j\lesssim \delta_m^d.
\]
Assume that the kernel satisfies the localization estimate \eqref{equ:localization} with $L>d+1$.
Then
\begin{equation}\label{equ:uniformstability}
\|\mathcal F_n f\|_{L^\infty}
\le C
\|f\|_{L^\infty},
\end{equation}
uniformly in $m$ and $n$, provided
\begin{equation}\label{equ:3.2condition}
n\lesssim \delta_m^{-1}.
\end{equation}
\end{theorem}
\begin{remark}
For \(0<\delta_m\le 1\), the balancing choice
\eqref{equ:optimalchoice} automatically satisfies
\eqref{equ:3.2condition}, since
\[
\frac{r}{s+r}<1.
\]
\end{remark}

\begin{proof}
For \(x\in\mathbb S^d\), we have
\[
|\mathcal F_n f(x)|
\le
\|f\|_{L^\infty}
\sum_{j=1}^m w_j|\Phi_n(x,x_j)|.
\]
By the localization estimate \eqref{equ:localization},
\[
|\Phi_n(x,x_j)|
\lesssim n^d(1+n\rho(x,x_j))^{-L}.
\]
Hence
\[
|\mathcal F_n f(x)|
\lesssim\|f\|_{L^\infty}
\sum_{j=1}^m
w_j n^d(1+n\rho(x,x_j))^{-L}.
\]
Since \(w_j\lesssim \delta_m^d\), it remains to estimate
\[
\sum_{j=1}^m (1+n\rho(x,x_j))^{-L}.
\]
For each integer \(k\ge0\), define
\[
A_k
:=
\{j:\ k\le n\rho(x,x_j)<k+1\}.
\]
The annuli $A_k$ form a disjoint partition of the index set
$\{1,\dots,m\}$. Hence
\[
\sum_{j=1}^m (1+n\rho(x,x_j))^{-L}
=
\sum_{k=0}^{\infty}
\sum_{j\in A_k}
(1+n\rho(x,x_j))^{-L}.
\]
If $j\in A_k$, then
\[
1+n\rho(x,x_j)\ge 1+k.
\]
Since \(L>0\), this implies
\[
(1+n\rho(x,x_j))^{-L}
\le
(1+k)^{-L}.
\]
Consequently,
\[
\sum_{j\in A_k}
(1+n\rho(x,x_j))^{-L}
\le
|A_k|(1+k)^{-L}.
\]
Summing over \(k\ge0\), we obtain
\[
\sum_{j=1}^m (1+n\rho(x,x_j))^{-L}
\le
\sum_{k=0}^{\infty}
|A_k|(1+k)^{-L}.
\]
By quasi-uniformity, the number of sampling points in a geodesic ball
\(B(x,R)\) is bounded by
\[
\#\bigl(\mathcal{X}_m\cap B(x,R)\bigr)
\lesssim\left(1+\frac{R}{\delta_m}\right)^d .
\]
Since \(A_k\subset B(x,(k+1)/n)\), we obtain
\[
|A_k|
\lesssim\left(1+\frac{k+1}{n\delta_m}\right)^d.
\]
Using \(n\lesssim\delta_m^{-1}\), we have \(n\delta_m\lesssim1\), and therefore
\[
1+\frac{k+1}{n\delta_m}
\lesssim\frac{k+1}{n\delta_m}.
\]
Thus
\[
|A_k|
\lesssim(n\delta_m)^{-d}(1+k)^d.
\]

It follows that
\[
\sum_{j=1}^m (1+n\rho(x,x_j))^{-L}
\lesssim(n\delta_m)^{-d}
\sum_{k=0}^\infty (1+k)^{d-L}.
\]
Since \(L>d+1\), the series converges. Hence
\[
\sum_{j=1}^m (1+n\rho(x,x_j))^{-L}
\lesssim(n\delta_m)^{-d}.
\]
Consequently,
\[
\sum_{j=1}^m w_j|\Phi_n(x,x_j)|
\lesssim n^d\delta_m^d (n\delta_m)^{-d}
\le
C.
\]
Therefore
\[
|\mathcal F_n f(x)|
\le
C\|f\|_{L^\infty}.
\]
Taking the supremum over \(x\in\mathbb S^d\) proves $\|\mathcal F_n f\|_{L^\infty}
\le C\|f\|_{L^\infty}$.
\end{proof}

\begin{remark}
The stability theorem shows that spectral localization stabilizes the discrete approximation operator itself. This is precisely the mechanism that fails for sharp spectral projection kernels in classical hyperinterpolation.
\end{remark}

\section{Examples of spectral multipliers}\label{sec:examples_of_spectrally_localized_kernels}

In this section we verify the assumptions in Section \ref{sec:main_results} for several important classes of spectral multipliers. The main point is that sufficiently smooth spectral multipliers simultaneously provide:

\begin{enumerate}
\item approximation property (A2),
\item spectral smoothing (A3),
\item (for filtered version) spatial localization \eqref{equ:localization} of the associated kernel.
\end{enumerate}

\subsection{Sharp spectral projection}

We first consider the classical sharp spectral projection corresponding to the
multiplier
\[
\eta_n(\lambda_\ell)
=
\mathbf 1_{\{\ell\le n\}}.
\]
In this case,
\[
\mathcal{T}_n=\mathcal{P}_n,
\]
the orthogonal projection \eqref{equ:orthogonalprojection} onto the spherical polynomial space \(\Pi_n\), and the
corresponding discrete approximation operator $\mathcal{F}_n$ reduces to the classical
hyperinterpolation operator \eqref{equ:hyperinterpolation}.

We first verify the approximation property, which can be also found in \cite{MR2274179}.
\begin{proposition}
Let \(0\le q < s\) and \(f\in H^s(\mathbb{S}^d)\). Then
\[
\|f-\mathcal{P}_n f\|_{H^q}
\lesssim n^{-(s-q)} \|f\|_{H^s}.
\]
\end{proposition}

\begin{proof}
Since
\[
f-\mathcal{P}_n f
=
\sum_{\ell>n}\sum_{k=1}^{Z(d,\ell)}
\widehat f_{\ell,k} Y_{\ell,k},
\]
we obtain
\[
\|f-\mathcal{P}_n f\|_{H^q}^2
=
\sum_{\ell>n}
(1+\lambda_\ell)^q
\sum_{k=1}^{Z(d,\ell)}
|\widehat f_{\ell,k}|^2.
\]
Since \(\lambda_\ell\sim \ell^2\), for \(\ell>n\) we have
\[
(1+\lambda_\ell)^q
\lesssim n^{-2(s-q)}
(1+\lambda_\ell)^s.
\]
Therefore,
\[
\|f-\mathcal{P}_n f\|_{H^q}^2
\lesssim n^{-2(s-q)}
\|f\|_{H^s}^2,
\]
which proves the result.
\end{proof}

We next verify the smoothing estimate.

\begin{proposition}
Let \(r>0\) and \(q\geq 0\). Then
\[
\|\mathcal{P}_n u\|_{H^q}
\lesssim n^{r+q}
\|u\|_{H^{-r}}.
\]
\end{proposition}

\begin{proof}
Since \(\mathcal{P}_n u\) is supported in frequencies \(\ell\le n\),
\[
\|\mathcal{P}_n u\|_{H^q}^2
=
\sum_{\ell\le n}
(1+\lambda_\ell)^q
\sum_{k=1}^{Z(d,\ell)}
|\widehat u_{\ell,k}|^2.
\]
For \(\ell\le n\),
\[
(1+\lambda_\ell)^q
\lesssim n^{2(r+q)}
(1+\lambda_\ell)^{-r}.
\]
Hence
\[
\|\mathcal{P}_n u\|_{H^q}^2
\lesssim n^{2(r+q)}
\|u\|_{H^{-r}}^2,
\]
which proves the estimate.
\end{proof}

Consequently, Theorem \ref{thm:main} applies also to the classical hyperinterpolation
operator under the weak cubature assumption (A1).

\begin{corollary}
Let \(0\le q < s\) and \(s>d/2+r\), and assume that the sampling measure satisfies (A1). Then
\[
\|f-\mathcal{L}_n f\|_{H^q}
\lesssim
\left(
n^{-(s-q)}
+
\delta_m^r n^{r+q}
\right)
\|f\|_{H^s}.
\]
The choice \eqref{equ:optimalchoice} yields
\[
\|f-\mathcal{L}_n f\|_{H^q}
\lesssim
\delta_m^{\frac{r(s-q)}{s+r}}
\|f\|_{H^s}.
\]
\end{corollary}

\begin{remark}
We emphasize, however, that the sharp spectral projection kernel
does not possess the strong spatial localization properties \eqref{equ:localization} enjoyed by smooth
spectral multipliers. Consequently, one cannot in general expect
uniform \(L^\infty\)-stability analogous to Theorem \ref{thm:uniformstability}.
Indeed, the corresponding Lebesgue constants typically grow with \(n\),
reflecting the poor localization of the sharp spectral projector; see, e.g., \cite{MR1761902,zbMATH01421286}.
\end{remark}

\subsection{Compactly supported smooth filters}

We then consider the classical filtered hyperinterpolation setting in \cite{MR2875103,sloan2012filtered}.
Let
\[
h\in C^\kappa([0,\infty)),
\qquad
\kappa>d+1,
\]
satisfy
\begin{equation*}
h(t)=\begin{cases}
1, & 0\le t\le1,\\
0, & t\ge 2.
\end{cases}
\end{equation*}
Define
\begin{equation}\label{equ:filteredkernel}
\eta_n(\lambda_\ell)=h(\ell/n).
\end{equation}
Then
\begin{equation}\label{equ:filteredhyperinterpolation}
\mathcal T_n f
=
\sum_{\ell=0}^\infty
h(\ell/n)
\sum_{k=1}^{Z(d,\ell)}
\widehat f_{\ell,k}Y_{\ell,k}.
\end{equation}
The multiplier is compactly supported in the sense that $h(\ell/n)=0$ whenever
$\ell\ge 2n$.
Thus $\mathcal T_n f$ is a spherical polynomial of degree at most $2n$.

We first verify the approximation property.
\begin{proposition}
Let $0\le q<s$ and $f\in H^s(\mathbb{S}^d)$. For the filtered kernel \eqref{equ:filteredkernel} and the associated spectral approximation operator \eqref{equ:filteredhyperinterpolation}, we have
\[
\|f-\mathcal T_n f\|_{H^q}
\lesssim n^{-(s-q)}
\|f\|_{H^s}.\]
\end{proposition}
\begin{proof}
Since $h(\ell/n)=1$ for $\ell\le n$,
we have
\[
f-\mathcal T_n f
=
\sum_{\ell>n}
\left(1-h(\ell/n)\right)
\sum_{k=1}^{Z(d,\ell)}
\widehat f_{\ell,k}Y_{\ell,k}.
\]
Therefore,
\[
\begin{aligned}
\|f-\mathcal T_n f\|_{H^q}^2
&=
\sum_{\ell>n}
(1+\lambda_\ell)^q
|1-h(\ell/n)|^2
\sum_{k=1}^{Z(d,\ell)}
|\widehat f_{\ell,k}|^2 .
\end{aligned}
\]
Since $h$ is bounded, $\lambda_\ell\sim \ell^2$, and for $\ell>n$ there holds
\[
(1+\lambda_\ell)^q
\lesssim n^{-2(s-q)}
(1+\lambda_\ell)^s,
\]
we have the desired estimate.
\end{proof}

We next verify the smoothing property. 
\begin{proposition}
Let $r>0$ and $q\geq 0$. For the filtered kernel \eqref{equ:filteredkernel} and the associated spectral approximation operator \eqref{equ:filteredhyperinterpolation}, we have
\[
\|\mathcal T_n u\|_{H^q}
\lesssim n^{r+q}
\|u\|_{H^{-r}}.
\]
\end{proposition}

\begin{proof}
Let $u\in H^{-r}(\mathbb S^d)$.
Since \(\mathcal T_n u\) is supported in frequencies \(\ell\le 2n\), we have
\[
\begin{aligned}
\|\mathcal T_n u\|_{H^q}^2
&=
\sum_{\ell\le 2n}
(1+\lambda_\ell)^q
|h(\ell/n)|^2
\sum_{k=1}^{Z(d,\ell)}
|\widehat u_{\ell,k}|^2 .
\end{aligned}
\]
For \(\ell\le 2n\),
\[
(1+\lambda_\ell)^q
\lesssim n^{2(r+q)}
(1+\lambda_\ell)^{-r}.
\]
Therefore,
\[
\|\mathcal T_n u\|_{H^q}^2
\lesssim n^{2(r+q)}
\sum_{\ell\le 2n}
(1+\lambda_\ell)^{-r}
\sum_{k=1}^{Z(d,\ell)}
|\widehat u_{\ell,k}|^2 ,
\]
and hence we obtain the smoothing estimate.
\end{proof}

We finally comment on the localization estimate \eqref{equ:localization}.
If the filter $h$ is compactly supported and possesses $\kappa$ continuous derivatives,
then the associated kernel satisfies a standard localization estimate of the form
\[
|\Phi_n(x,y)| \le C_L n^d (1+n\rho(x,y))^{-L}
\]
for every $L<\kappa$; see \cite{MR2673702,sloan2012filtered}.

Thus compactly supported smooth filters satisfy all assumptions in Section \ref{sec:main_results}. We obtain the following consequence.
\begin{corollary}[Compactly supported filtered hyperinterpolation]
Let
\[r>d/2,\qquad 
s>d/2 +r,\qquad
0\le q<s,\qquad
\kappa>d+1,\]
 and assume that the sampling measure satisfies (A1). Then
\[
\|f-\mathcal F_n f\|_{H^q}
\lesssim
\left(
n^{-(s-q)}
+
\delta_m^r n^{r+q}
\right)
\|f\|_{H^s}.
\]
The choice \eqref{equ:optimalchoice} yields
\[
\|f-\mathcal F_n f\|_{H^q}
\lesssim
\delta_m^{\frac{r(s-q)}{s+r}}
\|f\|_{H^s}.
\]
Moreover, under the conditions of Theorem \ref{thm:uniformstability}, the operator $\mathcal{F}_n$ has uniform stability in the sense of \eqref{equ:uniformstability}.
\end{corollary}
\subsection{Bessel potential multipliers}

We next consider Bessel potential multipliers. Let $\beta>0$
and define
\begin{equation}\label{equ:besselkernel}
\eta_n(\lambda_\ell)
=
\left(1+n^{-2}\lambda_\ell\right)^{-\beta/2}.
\end{equation}
The corresponding operator is
\begin{equation}\label{equ:besseloperator}
\mathcal T_n
=
\left(I-n^{-2}\Delta_{\mathbb S^d}\right)^{-\beta/2}.
\end{equation}

We first verify the approximation property. 

\begin{proposition}
Let $0\le s-q\le2$ and $f\in H^s(\mathbb{S}^d)$. For the Bessel multiplier \eqref{equ:besselkernel} and the associated spectral approximation operator \eqref{equ:besseloperator}, we have
\[
\|f-\mathcal T_n f\|_{H^q}
\lesssim n^{-(s-q)}
\|f\|_{H^s}.\]
\end{proposition}
\begin{proof}
For $f\in H^s(\mathbb S^d)$,
we have
\[
f-\mathcal T_n f
=
\sum_{\ell=0}^\infty
\left[
1-\left(1+n^{-2}\lambda_\ell\right)^{-\beta/2}
\right]
\sum_{k=1}^{Z(d,\ell)}
\widehat f_{\ell,k}Y_{\ell,k}.
\]
Therefore
\[
\|f-\mathcal T_n f\|_{H^q}^2
=
\sum_{\ell=0}^\infty
(1+\lambda_\ell)^q
\left|
1-\left(1+n^{-2}\lambda_\ell\right)^{-\beta/2}
\right|^2
\sum_{k=1}^{Z(d,\ell)}
|\widehat f_{\ell,k}|^2 .
\]
For \(0\le \theta\le1\), the elementary inequality
\[
\left|1-(1+t)^{-\beta/2}\right|
\le
C_{\beta,\theta} t^\theta,
\qquad t\ge0,
\]
holds, where $C_{\beta,\theta}>0$ is a constant depending on $\beta$ and $\theta$. Taking
\[
\theta=\frac{s-q}{2},
\]
which requires
\[
0\le s-q\le2,
\]
we obtain
\[
\left|
1-\left(1+n^{-2}\lambda_\ell\right)^{-\beta/2}
\right|
\lesssim n^{-(s-q)}\lambda_\ell^{(s-q)/2}.
\]
Consequently,
\[
(1+\lambda_\ell)^q
\left|
1-\left(1+n^{-2}\lambda_\ell\right)^{-\beta/2}
\right|^2
\lesssim n^{-2(s-q)}
(1+\lambda_\ell)^s.
\]
It follows that $\|f-\mathcal T_n f\|_{H^q}
\lesssim n^{-(s-q)}
\|f\|_{H^s}$
for 
$0\le s-q\le2$.
\end{proof}

We then verify the smoothing property. 

\begin{proposition}
Let $r>0$ and $q\geq 0$. For the Bessel multiplier \eqref{equ:besselkernel} and the associated  spectral approximation operator \eqref{equ:besseloperator}, if $\beta\ge r+q$, then we have
\[
\|\mathcal T_n u\|_{H^q}
\lesssim n^{r+q}
\|u\|_{H^{-r}}.
\]
\end{proposition}

\begin{proof}
Let $u\in H^{-r}(\mathbb S^d)$.
Then
\[
\|\mathcal T_n u\|_{H^q}^2
=
\sum_{\ell=0}^\infty
(1+\lambda_\ell)^q
\left(1+n^{-2}\lambda_\ell\right)^{-\beta}
\sum_{k=1}^{Z(d,\ell)}
|\widehat u_{\ell,k}|^2 .
\]
We claim that, provided
$\beta\ge r+q$,
one has
\[
(1+\lambda_\ell)^q
\left(1+n^{-2}\lambda_\ell\right)^{-\beta}
\lesssim n^{2(r+q)}
(1+\lambda_\ell)^{-r}.
\]
Indeed, this is equivalent to
\[
(1+\lambda_\ell)^{r+q}
\left(1+n^{-2}\lambda_\ell\right)^{-\beta}
\lesssim n^{2(r+q)}.
\]
Since
\[
1+n^{-2}\lambda_\ell
\ge
n^{-2}(1+\lambda_\ell)
\]
for \(n\ge1\), we have
\[
\left(1+n^{-2}\lambda_\ell\right)^{-(r+q)}
\lesssim n^{2(r+q)}
(1+\lambda_\ell)^{-(r+q)}.
\]
If \(\beta\ge r+q\), then
\[
\left(1+n^{-2}\lambda_\ell\right)^{-\beta}
\le
\left(1+n^{-2}\lambda_\ell\right)^{-(r+q)}.
\]
Combining these estimates gives the claim. Hence
$\|\mathcal T_n u\|_{H^q}
\lesssim n^{r+q}
\|u\|_{H^{-r}}$.
\end{proof}

Finally, the Bessel multiplier \eqref{equ:besselkernel}
is a smooth spectral multiplier whose derivatives satisfy
polynomial decay estimates depending on \(\beta\).
Standard localization results (e.g. \cite{MR2673702}) for smooth spectral multipliers imply that, if \(\beta\) is sufficiently large, then the associated kernel satisfies \eqref{equ:localization} for some \(L>d+1\).

We therefore obtain the following consequence.

\begin{corollary}[Bessel potential filtered hyperinterpolation]
Let
\[
r>d/2,\qquad
s> d/2+r, \qquad 
0\le q<s,
\qquad
0<s-q\le2.
\]
Assume that
\[
\beta\ge r+q
\]
and that \(\beta\) is sufficiently large so that the associated Bessel kernel satisfies the required localization estimate \eqref{equ:localization} in the sense that $L>d+1$ is admissible. If the sampling measure satisfies (A1),
then
\[
\|f-\mathcal F_n f\|_{H^q}
\lesssim
\left(
n^{-(s-q)}
+
\delta_m^r n^{r+q}
\right)
\|f\|_{H^s}.
\]
The choice \eqref{equ:optimalchoice} yields
\[
\|f-\mathcal F_n f\|_{H^q}
\lesssim
\delta_m^{\frac{r(s-q)}{s+r}}
\|f\|_{H^s}.
\]
Moreover, under the conditions of Theorem \ref{thm:uniformstability}, the operator $\mathcal{F}_n$ has uniform stability in the sense of \eqref{equ:uniformstability}.
\end{corollary}

\begin{remark}
The restriction
\[
s-q\le 2
\]
is intrinsic to the Bessel-type multiplier
\[(1+n^{-2}\lambda)^{-\beta/2}.\]
Although increasing \(\beta\) improves smoothing and kernel localization,
the multiplier still satisfies
\[
1-(1+n^{-2}\lambda)^{-\beta/2}
=
O(n^{-2}\lambda)
,\qquad \lambda\to0.
\]
Thus the leading consistency error remains of order \(n^{-2}\),
and no higher-order approximation can be achieved solely by increasing \(\beta\).
\end{remark}

\subsection{Heat kernel multipliers}

We finally consider heat kernel multipliers. Define
\begin{equation}\label{equ:heatkernel}
\eta_n(\lambda_\ell)
=
e^{-\lambda_\ell/n^2}.
\end{equation}
The corresponding operator is the heat semigroup
\begin{equation}\label{equ:heatoperator}
\mathcal T_n
=
e^{n^{-2}\Delta_{\mathbb S^d}}.
\end{equation}

We first verify the approximation property. 
\begin{proposition}
Let $0\le s-q\le2$ and $f\in H^s(\mathbb{S}^d)$. For the heat multiplier \eqref{equ:heatkernel} and the associated spectral approximation operator \eqref{equ:heatoperator}, we have
\[
\|f-\mathcal T_n f\|_{H^q}
\lesssim n^{-(s-q)}
\|f\|_{H^s}.\]
\end{proposition}
\begin{proof}
For $f\in H^s(\mathbb S^d)$, we have
\[
f-\mathcal T_n f
=
\sum_{\ell=0}^\infty
\left(1-e^{-\lambda_\ell/n^2}\right)
\sum_{k=1}^{Z(d,\ell)}
\widehat f_{\ell,k}Y_{\ell,k}.
\]
Thus
\[
\|f-\mathcal T_n f\|_{H^q}^2
=
\sum_{\ell=0}^\infty
(1+\lambda_\ell)^q
\left|1-e^{-\lambda_\ell/n^2}\right|^2
\sum_{k=1}^{Z(d,\ell)}
|\widehat f_{\ell,k}|^2 .
\]
For \(0\le \theta\le1\), the elementary inequality
\[
1-e^{-t}
\le
C t^\theta,
\qquad t\ge0,
\]
holds. Taking
\[
\theta=\frac{s-q}{2},
\]
which requires
\[
0\le s-q\le2,
\]
we obtain
\[
1-e^{-\lambda_\ell/n^2}
\lesssim n^{-(s-q)}
\lambda_\ell^{(s-q)/2}.
\]
Consequently,
\[
(1+\lambda_\ell)^q
\left|1-e^{-\lambda_\ell/n^2}\right|^2
\lesssim n^{-2(s-q)}
(1+\lambda_\ell)^s.
\]
Therefore, $\|f-\mathcal T_n f\|_{H^q}
\lesssim n^{-(s-q)}
\|f\|_{H^s}$ for
$0\le s-q\le2$.
\end{proof}

We next verify the smoothing property.

\begin{proposition}
Let $r>0$ and $q\geq 0$. For the heat multiplier \eqref{equ:heatkernel} and the associated spectral approximation operator \eqref{equ:heatoperator}, we have
\[
\|\mathcal T_n u\|_{H^q}
\lesssim n^{r+q}
\|u\|_{H^{-r}}.
\]
\end{proposition}
\begin{proof}
 Let $u\in H^{-r}(\mathbb S^d)$.
Then
\[
\|\mathcal T_n u\|_{H^q}^2
=
\sum_{\ell=0}^\infty
(1+\lambda_\ell)^q
e^{-2\lambda_\ell/n^2}
\sum_{k=1}^{Z(d,\ell)}
|\widehat u_{\ell,k}|^2 .
\]
We claim that
\[
(1+\lambda_\ell)^q e^{-2\lambda_\ell/n^2}
\lesssim n^{2(r+q)}
(1+\lambda_\ell)^{-r}.
\]
Indeed, it is enough to prove
\[
(1+\lambda_\ell)^{r+q}e^{-2\lambda_\ell/n^2}
\lesssim n^{2(r+q)},
\]
which follows from the following estimate
\[
a^{r+q}e^{-ca/n^2}
\lesssim n^{2(r+q)},
\qquad a\ge0.
\]
 Hence $\|\mathcal T_n u\|_{H^q}
\lesssim n^{r+q}
\|u\|_{H^{-r}}$.
\end{proof}

Finally, the heat kernel on the sphere satisfies the Gaussian upper bound
\[
|\Phi_n(x,y)|
\le C n^d \exp(-c n^2\rho(x,y)^2),
\]
which is the standard heat kernel estimate with time parameter \(t=n^{-2}\).
Such Gaussian estimates are among the basic assumptions and consequences
in the generalized heat kernel framework of Filbir and Mhaskar
\cite{MR2673702}. Since
\[
e^{-c u^2}\le C_L(1+u)^{-L}, \qquad u\ge0,
\]
for every \(L>0\), it follows that
\[
|\Phi_n(x,y)|
\le C_L n^d(1+n\rho(x,y))^{-L}.
\]
Thus the localization requirement in Theorem \ref{thm:uniformstability} is satisfied for arbitrary
polynomial decay order.

We therefore obtain the following consequence.

\begin{corollary}[Heat filtered hyperinterpolation]
Let
\[
r>d/2,\qquad
s>d/2+r,\qquad
0\le q<s,
\qquad
0<s-q\le2.
\]
Assume that the sampling measure satisfies (A1).
Then
\[
\|f-\mathcal F_n f\|_{H^q}
\lesssim
\left(
n^{-(s-q)}
+
\delta_m^r n^{r+q}
\right)
\|f\|_{H^s}.
\]
The choice \eqref{equ:optimalchoice} yields
\[
\|f-\mathcal F_n f\|_{H^q}
\lesssim
\delta_m^{\frac{r(s-q)}{s+r}}
\|f\|_{H^s}.
\]
Moreover, under the conditions of Theorem \ref{thm:uniformstability}, the operator $\mathcal{F}_n$ has uniform stability in the sense of \eqref{equ:uniformstability}.
\end{corollary}

\begin{remark}
A similar saturation phenomenon occurs for the heat-kernel multiplier
\[
\eta_n(\lambda)=e^{-\lambda/n^2}.
\]
Although this multiplier yields Gaussian localization of the associated kernel,
its consistency near the origin is still only of first order in \(n^{-2}\lambda\), since
\[
1-e^{-\lambda/n^2}
=
O(n^{-2}\lambda),
\qquad \lambda\to0.
\]
Consequently, the approximation error saturates at order \(n^{-2}\).
Thus stronger kernel localization does not by itself imply higher approximation order;
the latter is determined by the low-frequency behavior of the multiplier near
\(\lambda=0\).
\end{remark}

\section{Connection with spherical QMC designs}\label{sec:connection_with_spherical_qmc_designs}

The weak cubature condition (A1) considered in this paper may be interpreted as a Sobolev discrepancy estimate for the sampling measure.
Indeed, for
$\nu_m=\mu_m-\sigma$, one has
\begin{equation*}\|\nu_m\|_{H^{-r}(\mathbb S^d)}
=
\sup_{\|g\|_{H^r}\le1}
\left|
\int_{\mathbb S^d}g\,d\sigma
-
\sum_{j=1}^m
w_jg(x_j)
\right|.
\end{equation*}
Thus the weak cubature assumption is precisely a worst-case integration error estimate over the Sobolev space
$H^r(\mathbb S^d)$. This condition is closely related
to the theory of QMC designs, proposed in \cite{MR3246811}.
A sequence of point sets $\{\mathcal{X}_m\}_{m}$
is called a sequence of QMC designs for \(H^r(\mathbb{S}^d)\),
with \(r>d/2\), if there exists a constant \(C>0\) independent of \(m\) such that
\begin{equation}\label{equ:QMC}
\sup_{\|f\|_{H^r}\le1}\left|
\frac1m\sum_{j=1}^m f(x_j)
-
\int_{\mathbb S^d} f\,d\sigma
\right|
\le
C m^{-r/d}.
\end{equation}
In the equal-weight case
\[
w_j=\frac1m,
\]
the discrepancy measure becomes
\[
\nu_m
=
\frac1m\sum_{j=1}^m\delta_{x_j}-\sigma.
\]
The weak cubature condition (A1) can be realized by QMC designs in the sense of
\[
\|\nu_m\|_{H^{-r}}
\le
C m^{-r/d}.
\]

Consequently, Theorem \ref{thm:main} immediately yields the following result.
\begin{corollary}
Let \(r>d/2\), \(0\le q<s\), and let $\{\mathcal{X}_m\}$ be a quasi-uniform sequence of QMC designs for \(H^r(\mathbb{S}^d)\) and $w_j=1/m$. 
Then for the spectral multiplier hyperinterpolation operator \eqref{equ:spectrallyfilteredhyperinterpolation},
there holds
\[
\|f-\mathcal F_n f\|_{H^q}
\lesssim
\left(
n^{-(s-q)}
+
m^{-r/d} n^{r+q}
\right)
\|f\|_{H^s}.
\]
The corresponding balancing choice
\[
n\sim m^{r/(d(s+r))}
\]
yields 
\[
\|f-\mathcal F_n f\|_{H^q}
\lesssim
m^{-\frac{r(s-q)}{d(s+r)}}
\|f\|_{H^s}.
\]
\end{corollary}

\begin{remark}
For the sharp spectral projection, the approximation rates obtained here with $q=0$
are essentially consistent with those in our previous work
\cite[Section 4]{an2024bypassing},
where hyperinterpolation based on QMC designs was analyzed primarily in the \(L^2\) setting.
\end{remark}

\begin{corollary}
Assume the kernel \(\Phi_n\) satisfies \eqref{equ:localization}
for some \(L>d+1\).
Let \(r>d/2\), and 
let $\{\mathcal{X}_m\}$ be a quasi-uniform sequence of QMC designs for \(H^r(\mathbb{S}^d)\) and $w_j=1/m$. Assume additionally that
$n\lesssim m^{1/d}$.
Then the corresponding spectral multiplier hyperinterpolation operator \eqref{equ:spectrallyfilteredhyperinterpolation}
satisfies
\[
\|\mathcal F_n f\|_{L^\infty}
\le
C
\|f\|_{L^\infty}.
\]
\end{corollary}

\begin{remark}
This extends the theory at the interface between filtered hyperinterpolation, more general spectral multiplier methods, and QMC designs.\end{remark}

\section{Concluding remarks}

The present work proposes a weak-cubature framework for hyperinterpolation on the sphere based on Sobolev discrepancy estimates, extending classical theories that rely on exact polynomial cubature formulas or MZ inequalities. The central idea is to interpret the discretization error through the action of spectral multiplier operators on the cubature discrepancy measure. This perspective separates the approximation problem into three distinct components: approximation properties of the underlying continuous spectral operator, smoothing and localization properties of the spectral multiplier, and discrepancy estimates of the sampling measure.

From this viewpoint, stable Sobolev approximation may be achieved under substantially weaker assumptions than exact polynomial reproduction. Moreover, the framework applies uniformly to a broad family of spectral approximation operators, including sharp spectral projections, compactly supported smooth filters, Bessel potential operators, and heat kernel operators. For sufficiently localized spectral multipliers, the resulting discrete approximation operators additionally enjoy uniform \(L^\infty\)-stability.

The results also highlight a natural connection between hyperinterpolation, Sobolev discrepancy, and QMC designs. In particular, the weak cubature assumptions considered here may be interpreted in terms of worst-case integration errors in Sobolev spaces, thereby linking approximation from scattered data with contemporary QMC design theory.

\bibliographystyle{siamplain}
\bibliography{myref}
\clearpage

\end{document}